\documentclass{amsart}
\usepackage{graphicx,xfrac} 
\usepackage{amssymb} 
\usepackage[all]{xy}
\usepackage{amsmath,amsthm}
\usepackage{todonotes}
\usepackage{comment}
\usepackage{tikz}

\newcommand{\cO}{\mathcal{O}}
\newcommand{\cL}{\mathcal{L}}
\newcommand{\cF}{\mathcal{F}}
\newcommand{\cE}{\mathcal{E}}

\newcommand{\bP}{\mathbb{P}}
\newcommand{\bF}{\mathbb{F}}

\theoremstyle{plain}
\newtheorem{theorem}{Theorem}
\newtheorem{proposition}[theorem]{Proposition}

\newtheorem{lemma}[theorem]{Lemma}

\theoremstyle{definition}

\newtheorem{example}[theorem]{Example}
\newtheorem{remark}[theorem]{Remark}

\title{On the instability of syzygy bundles on toric surfaces}

\author[L. Devey]{Lucie Devey}
\address{School of Mathematics,
The University of Edinburgh, United Kingdom}
\email{ldevey2@ed.ac.uk}

\author[M. Hering]{Milena Hering}
\address{School of Mathematics,
The University of Edinburgh, United Kingdom}
\email{m.hering@ed.ac.uk}

\author[K. Jochemko]{Katharina Jochemko}
\address{KTH Royal Institute of Technology, Stockholm, Sweden}
\email{jochemko@kth.se}

\author[H. S\"uß]{Hendrik S\"uß}
\address{Institut f\"ur Mathematik, Friedrich-Schiller-Universität Jena, Germany}
\email{hendrik.suess@uni-jena.de}

\begin{document}

\begin{abstract}
    We show that for every toric surface $X$ apart from $\mathbb{P}^2$ and $\mathbb{P}^1\times\mathbb{P}^1$ and every ample line bundle $\cL$ on $X$ there exists an ample polarisation $A$ for $X$, such that the syzygy bundle $M_{\cL^{\otimes d}}$ associated to the tensor power $\cL^{\otimes d}$ is not stable with respect to $A$ for every $d$ sufficiently large. 
\end{abstract}

\maketitle

Let $X$ be a projective variety of dimension $n$ over an algebraically closed field $k$ and let  $\mathcal{L}$ be an ample and globally generated line bundle on $X$. The \emph{syzygy bundle} $M_{\mathcal{L}}$ is defined by the following short exact sequence: 
\begin{equation}\label{eq:ML}
0 \to M_{\mathcal{L}} \to H^0(X,\mathcal{L})\otimes \mathcal{O}_X \to \mathcal{L}\to 0.
\end{equation}

There has been interest in (slope) stability of syzygy bundles as they come up in the study of defining equations and higher syzygies of embeddings of algebraic varieties, and there also are connections to tight closure. The study of syzygy bundles goes back to Butler~\cite{Butler}.  On curves,  $M_L$ is stable as soon as $\deg L\geq g+1$, and several cases are known in higher dimensions, see for example the introduction to \cite{EinLazarsfeldMustopa}.
More recently, Rekuski \cite{Rekuski} showed that  for an ample and globally generated line bundle $\mathcal{L}$ on a smooth projective variety, $M_{\cL^d}$ is stable with respect to $\cL$ for $d\gg 0$ and provides effective bounds for $d$, generalising a result on surfaces by \cite{EinLazarsfeldMustopa}. 

In \cite[Problem 2.5]{EinLazarsfeldMustopa}, the authors ask whether for $\cL$ ample and $A$ an arbitrary ample divisor,  $M_{\cL^d}$ is  slope stable with respect to $A$ for $d\gg 0$. The following proposition shows that this is not the case. It also provides a counterexample to Theorem~1.2 of \cite{TLZ}.

\begin{proposition}\label{prop:counterexample1}
    Let $X$ be the blow up of the projective plane at one point, $E$ the exceptional divisor and $H$ the hyperplane divisor. Then for $D = 6H-E$ and $d > 17$ we have that $M_{\mathcal{O}(dD)}$ is not semistable with respect to $-K_X = 3H-E$. The destabilising subbundle is $M_{\mathcal{O}(dD-E)}$.
\end{proposition}

We obtain the following generalization to almost all toric surfaces. 

\begin{theorem}\label{thm:toricsurface}
Let $X$ be a smooth toric surface not isomorphic to $\mathbb{P}^2$ and $\mathbb{P}^1\times\mathbb{P}^1$. Then for every ample divisor $D$ there exists an ample polarisation $A$ and a positive integer $d_0$ such that for $d\geq d_0$ we have $M_{O(dD)}$ is not stable with respect to $A$. 
\end{theorem}

For an ample divisor $D$ on a Hirzebruch surface, we can precisely describe the set of ample polarizations such that for $d\gg 0$, $M_{\mathcal{O}(dD)}$ is not stable with respect to $A$. 
Let $\bF_{\ell} \to \bP^1$ be the $\ell$-th Hirzebruch surface with $F$ being a fiber  and $S$ the unique section with self-intersection $-\ell$. 
\begin{theorem}\label{thm:hirzebruch}
Let $\bF_{\ell}$ be a Hirzebruch surface with  $\ell\geq 1$ and  $A=A_1S+A_2F$ and $D=B_1S+B_2F$  be ample divisors on $\bF_{\ell}$. Set $a=A_2/A_1$ and  $b=B_2/B_1$. Suppose 
    \begin{equation}\label{eq:stab-condition-hirzebruch}
a >\frac{2b(b-\ell)}{\ell} +\ell \text{ or equality holds and } b\geq \frac{3\ell}{4} + \sqrt{\frac{\ell^2}{16}+\frac{\ell}{2}}.
\end{equation}

Then for $d\gg 0$, $M_{O(dD)}$ is not stable with respect to $A$. The destabilising subbundle is given by $M_{O(dD-S)}$. 
\end{theorem}

We will see that the destabilizing subbundles are of the form $M_{\cO(D')}$, where $D'$ is nef and globally generated, and $D-D'$ is effective. In a forthcoming article we will show that on a toric variety the subbundle of maximal slope of $M_{\cO(D)}$ for $D$ ample is of this form. Generalised versions of these subbundles appeared e.g. already in \cite[1.9]{Butler}.

\noindent
  \textbf{Acknowledgements:} We would like to thank Ana-Maria Castravet for pointing out \cite[Prop 2.3]{cltu}.
  LD was partially supported by EPSRC fellowship  EP/T018836/1. MH was partially supported by EPSRC fellowship  EP/T018836/1 and an Emmy Nother Fellowship. 
  KJ was partially supported by the Wallenberg AI, Autonomous Systems and Software Program funded by the Knut and Alice Wallenberg Foundation as well as grant nr. 2018-03968 of the Swedish Research Council. HS was partially supported by EPSRC grants EP/V013270/1,  EP/V055445/1 and by the Carl Zeiss Foundation.

\section*{Proofs of the theorems}

Recall that 
the slope of a reflexive sheaf $\mathcal{F}$ with respect to a polarisation given by an ample divisor $A$ is given by 
\[ \mu_A(\mathcal{F})=\frac{c_1(\mathcal{F})\cdot A^{n-1}}{\mathrm{rk}{F}}.\]

We call a vector bundle $\cE$ stable (semistable) with respect to $A$ if for every reflexive subsheaf $\cF$ of $\cE$, we have $\mu_A(\mathcal{F})<\mu_A(\mathcal{E})$
($\mu_A(\mathcal{F})\leq \mu_A(\mathcal{E})$).

Now let $\mathcal{L} = O(D)$ for some ample divisor $D$. Then it follows from the short exact sequence \eqref{eq:ML} and additivity of first chern classes that $c_1(M_{\mathcal{O}(D)}) =-D$. So 
\begin{equation}\label{eq:slope}
\mu_A(M_{\mathcal{O}(D)}) = \frac{-D\cdot A^{n-1}}{h^0(X,D)-1}.
\end{equation}

\begin{lemma}
Let $D$ be ample and $D'$ nef such that  $D-D'$ is effective and assume that $\cO(D), \cO(D')$ are globally generated. Then $M_{\cO(D')}$ is naturally a subbundle of $M_{\cO(D)}$. 
\end{lemma}
\begin{proof}
This follows from the following diagram. 
\begin{equation}\label{Intro.Diagram}
\begin{gathered}
\xymatrix{
0\ar[r]& M_{\cO(D')}\ar[r] & H^0(X,\cO(D'))\otimes \cO_X \ar@{^{(}->}[d] \ar[r]&\cO(D')\ar@{^{(}->}[d]\ar[r]&0\\ 0 \ar[r] & M_{\cO(D)}\ar[r] &  H^0(X,\cO(D))\otimes \cO_X \ar[r] & \cO(D) \ar[r]&0. 
}
\end{gathered}
\end{equation}
Indeed, the square on the right commutes, which implies the claim.
\end{proof}

\begin{proof}[Proof of Proposition~\ref{prop:counterexample1}]
    Recall that $E^2=-1, H^2=1$ and $E\cdot H=0$. For an arbitrary divisor $D'=mH+nE$, we have $(D')^2 = m^2-n^2$ and $D'\cdot (-K_X) = 3m+n$. Using the formula for the slope \eqref{eq:slope} and Riemann-Roch for surfaces
    \[h^0(\cO(D'))=1+\frac{(D')^2-D'\cdot K_X}{2}, \]
    we compute the slope 
    \begin{equation}
        \mu_{-K_X}(M_{\cO(D')}) = \frac{-6m-2n}{m^2-n^2+3m+n}.
    \end{equation}
    For $m=6d$, $n=-d$ and $m=6d$, $n=-d-1$, respectively, we compute then that 
    \[ \mu_{-K_X}(M_{\cO(dD)}) = \frac{-34d}{35d^2+17d} \textrm{ and } \mu_{-K_X}(M_{\cO(dD-E)}) = \frac{-34d+2}{35d^2+15d-2}.\]
    It then follows that $\mu_{-K_X}(M_{\cO(dD-E)}) >  \mu_{-K_X}(M_{\cO(dD)})$ if and only if $d > 17$, so in that case $M_{\mathcal{O}(dD)}$ is not semistable. 
\end{proof}

\begin{remark}
  In a forthcoming paper, by exploiting the fact that $X$ is toric, we will show that the bundle $M_{\cO(D)}$ from Proposition \ref{prop:counterexample1} is in fact stable. Hence, this provides an example of an ample line bundle $\cL=\cO(D)$ and a polarization $A$ where $M_\cL$ is $A$-stable, but $M_{\cL^{\otimes d}}$ is unstable with respect to $A$ for $d\gg 0$.
\end{remark}

  \begin{proposition}
    \label{prop:asymptotic-rat-surface}
    Let $X$ be a rational surface and $D$ and $A$ ample divisors. Assume $M_{\cO(dD)}$ is  $A$-stable for $d\gg0$. Then  for all effective divisors $S$ we have
    \begin{equation}
    2(D\cdot A)(D\cdot S) - (S\cdot A)(D^2) \geq 0.\label{eq:stab-condition1}    
  \end{equation}
    and in case of equality we additionally have
    \begin{equation}
      \label{eq:stab-condition2}
      -(D\cdot A)(S^2+S\cdot K_X)+(S\cdot A)(D\cdot K_X) > 0.
    \end{equation}
  \end{proposition}
  \begin{proof}
    We may assume that $d$ is sufficiently large such that $dD-S$ and $dD-S-K_X$ are still ample.
    Set $\cL=\cO(dD)$, $\cL_1=\cO(dD-S)$. It follows from Kodaira Vanishing Theorem that the higher cohomology groups of $\cL$ and $\cL_1$ are all $0$. We consider the subbundle $\cF_1 = M_{\cL_1} \subset M_\cL$. Using this notation we obtain from \eqref{eq:slope}
    \begin{align}
  \mu_A(M_\cL)&=-\frac{dD \cdot A}{h^0(\cL)-1}= -\frac{dD \cdot A}{\frac{1}{2}\left(d^2D^2 - dD \cdot K_X\right)}\label{eq:muL}\\
  \mu_A(\cF_1)&=-\frac{(dD-S) \cdot A}{h^0(\cL_1)-1}=-  \frac{dD\cdot A- S\cdot A}{\frac{1}{2}\left((d^2 D^2 -d D \cdot K_X)+\left( -2dD\cdot S + S^2 + S \cdot K_X\right)\right)}. \label{eq:muF1}
\end{align}
Here, we are using the Riemann-Roch formula 
\[h^0(\cL)=\chi(\cL)=1+\frac{D^2-D\cdot K_X}{2}.\]

We are interested in the sign of $\mu_A(M_\cL)-\mu_A(\cF_1)$  for $d\gg 0$. We write $\mu_A(M_\cL)-\mu_A(\cF_1)$
as a fraction with common denominator $2(h^0(\cL)-1)(h^0(\cL_1)-1)$, which is clearly positive. The numerator of  $\mu_A(M_\cL)-\mu_A(\cF_1)$ is 
\begin{equation}
  \label{eq:mu-difference}
    \left(2(D\cdot A)(D\cdot S) - (S\cdot A)(D^2)\right)d^2
 + (-(D\cdot A)(S^2+S\cdot K_X)+(S\cdot A)(D\cdot K_X))d.
\end{equation}
Now, for $d \gg 0$ this sign is determined by the term of the highest degree in $d$, which 
implies (\ref{eq:stab-condition1}). In the case of equality the terms of $d$-degree $1$ determine the sign, which implies \eqref{eq:stab-condition2}.
  \end{proof}
  
Let us consider the Hirzebruch surface of degree $\ell$ that is $\bF_\ell=\operatorname{Proj}_{\bP^1}(\cO_{\bP^1} \oplus \cO_{\bP^1}(\ell))$. There is the structure morphism $\bF_\ell \to \bP^1$. Let $S\subset \bF_\ell$ denote  the unique section of self intersection $-\ell$ and let $F$ be a fibre. Then the Picard group and effective cone are generated by $S$ and $F$. The intersection product is given by $S^2=-\ell$, $F^2=0$ and $S\cdot F=1$. The canonical divisor is given by $K_{\bF_\ell}=-(2+\ell)F-2S$.

\begin{proof}[Proof of Theorem \ref{thm:hirzebruch}]
Let $A' = S+aF$ and $D'=S+bF$. Assume $M_{\cO(dD)}$ is $A$-stable for $d\gg 0$. Then  by Proposition \ref{prop:asymptotic-rat-surface} we must have 
\begin{align}
  2(D\cdot A)(D\cdot S) - (A\cdot S)(D^2)&=2(B_1D'\cdot A_1A')(B_1D'\cdot S) - (A_1A'\cdot S)((B_1D')^2)\\
                                         &=A_1B_1^2(2(a+b-\ell)(b-\ell)-(a-\ell)(2b-\ell))\nonumber \\
                                         &=A_1B_1^2(2b(b-\ell) +\ell^2-\ell a)\\
                                         &\geq 0       \nonumber                                    
\end{align}

Since $A_1, B_1$ are positive integers this implies 
\begin{equation}
a \leq \frac{2b(b-\ell)}{\ell} +\ell.\label{eq:stab-condition-hirzebruch1}
\end{equation}

Now, if $a = \frac{2b(b-\ell)}{\ell} +\ell$ then (\ref{eq:stab-condition2}) in Proposition~\ref{prop:asymptotic-rat-surface} has to hold. This leads to 
\[
 -(D\cdot A)(S^2+S\cdot K_X)+(S\cdot A)(D\cdot K_X)=A_1B_1(-(a+b-\ell)(-2)-(a-\ell)(2b+2-\ell)) > 0.
\]
Dividing by $A_1B_1$ and substituting $a$ by $\frac{2b(b-\ell)}{\ell} +\ell$ gives ${2  {\left({\left(3  b^{2} + b\right)} \ell -2b^{3} - b \ell^{2} \right)}}\cdot{\ell^{-1}} >0$ which is equivalent to $b^{2} - \frac{3}{2}\ell  b  +  \frac{1}{2}(\ell^{2}-\ell) < 0$.

The two roots of the polynomial on the left-hand-side are
\[\frac{3\ell}{4} \pm \sqrt{\frac{\ell^2}{16}+\frac{\ell}{2}}\]

Note, that since  $D$ is ample, we must have $b>\ell$ and the smaller root is in fact smaller than $\ell$. Hence, for any choice of $b > \ell$ and $a = \frac{2b(b-\ell)}{\ell} +\ell$ the line bundle $\cO(D)$ we must have
\begin{equation}
  \ell < b < \frac{3\ell}{4} + \sqrt{\frac{\ell^2}{16}+\frac{\ell}{2}}.\label{eq:stab-condition-hirzebruch2}
\end{equation}
\end{proof}

\begin{example}\label{ex:delPezzo}

Note that $\mathbb{F}_1$ is the blow up of $\mathbb{P}^2$ at a point from Propsition \ref{prop:counterexample1} and we have $S=E$ and $F=H-E$. 
Again, we study the stability with respect to $-K_{\bF_1}$ (or equivalently multiples thereof). For $\ell=1$ the condition (\ref{eq:stab-condition-hirzebruch1}) becomes $a \leq 2b(b-1)+1$.

  Now choose e.g. $b = \frac{9}{8}$. Then  the syzygy bundle $M_{\cL^{\otimes d}}$ with $\cL=\cO(8S+9F)$ is $(-K_{\bF_1})$-unstable for $d \gg 0$. Indeed, $-K_{\bF_1}=2S+3F$ and we have
  \[
    a=\frac{3}{2} = \frac{96}{64} > 2\cdot \frac{9}{8}\cdot \frac{1}{8}+1 =\frac{82}{64}=2b(b-1)+1.
  \]
  It is not hard to check that already for $d=1$ the corresponding syzygy bundle $M_\cL$ is $(-K_{\bF_1})$-unstable with $\cF_1=M_{\cO(7S+9F)}$ being a destabilising subbundle. Indeed, setting $D=8S+9F$, $A=-K_X=2S+3F$ and $d=1$ in (\ref{eq:muL}) and (\ref{eq:muF1}) leads to
  \[
    \mu_A(M_\cL)=-\frac{26}{53} <     \mu_A(M_{\cF_1})=-\frac{25}{51}.
  \]
  To obtain a necessary condition when a tensor power of $\cL=\cO(S+bF)$ can have $(-K_{\bF_1})$-stable syzygy bundle, we solve condition ~\eqref{eq:stab-condition-hirzebruch} in $b$ for $\ell=1$ and $a=\frac{3}{2}$, which leads to the condition $b \geq \frac{1}{2}(\sqrt{2} +1)$. 
  Note that this is in line with 
  \cite[Theorem 4.4]{Rekuski}, since the sufficient condition there is more restrictive than the necessary condition given above.
\end{example}

\begin{remark}
  It is natural to ask whether in the setting of Theorem \ref{thm:hirzebruch} for  $d \gg 0$ the bundle $M_{\cO(dD)}$ is in fact $A$-semistable if the  condition in (\ref{eq:stab-condition-hirzebruch1}) is fulfilled. In a forthcoming paper we will give a combinatorial proof that this is true. 
  \end{remark}

  \begin{proposition}
    \label{prop:rationalsurface}
  Let $X$ be a smooth surface of Picard number at least $3$ having an effective cone generated by curves  $E_1, \ldots, E_m$ with mutual intersection multiplicity of at most $1$. Then  for every ample divisor $D$ there exists an ample polarisation $A$ and a positive integer $d_0$ such that for $d\geq d_0$ the syzygy bundle $M_{O(dD)}$ is not stable with respect to~$A$. 
  \end{proposition}
  \begin{proof} 
    We consider $E_j$ with the property that $D\cdot E_j$ is minimal among the $E_i$.  Set $A=D-tE_j$, where $t$ is the nef-threshold, i.e. $t=\max \{t \mid D -tE_j \text{ is nef}\}$. Since $A$ is on the boundary of the nef cone, we have $A\cdot E_i=0$ for some generator $E_i$ and for all
    such $E_i$ one can deduce from the assumption $E_i\cdot E_j \in \{0,1\}$ that $D\cdot E_i=t$. By the choice of $E_j$ this implies 
\begin{equation}
t \geq D \cdot E_j.\label{eq:threshold-estimate}
\end{equation}

Now, we look at the condition (\ref{eq:stab-condition1}) in the statement of Proposition~\ref{prop:asymptotic-rat-surface} where we set $S=E_j$. 
By \cite[Prop. 2.3.]{cltu} we may assume that the curves $E_1, \ldots, E_m$ have negative self-intersection.
We then obtain 
\begin{align*}
  2(D\cdot A)(D\cdot S) - (S\cdot A)D^2 &= 2(D\cdot (D-tE_j))(D\cdot E_j) - (E_j\cdot (D-tE_j))D^2 \\
                                          &= D^2(D\cdot E_j) - 2t (D \cdot E_j)^2+t(E_j^2)D^2\\
                                          & < D^2(D\cdot E_j) +tE_j^2D^2\\
                                          & \leq  D^2(D\cdot E_j) - tD^2\\
                                          &= D^2(D\cdot E_j - t)\\
                                           & \leq 0.
\end{align*}
Where the last inequality follows from (\ref{eq:threshold-estimate}). Note, that $A$ is only nef, not ample. However, replacing $A$ by $A'=A+\epsilon E_j$ will be ample for $0 < \epsilon \leq 1$ and by continiuty we have
$2(D\cdot A')(D\cdot S) - (S\cdot A')D^2 < 0$ for $\epsilon$ sufficiently small. Now, our claim follows from Proposition~\ref{prop:asymptotic-rat-surface}.
  \end{proof}

  \begin{proof}[Proof of Theorem \ref{thm:toricsurface}]
    The case of Hirzebruch surfaces has been dealt with in Theorem~\ref{thm:hirzebruch}, so assume that $X$ is not a Hirzebruch surface, $\bP^1 \times \bP^1$ or $\bP^2$. 
    It follows from the classification of toric surfaces \cite[p.43]{Fulton93},  that in this case $X$ has Picard rank at least 3. The effective cone of any toric variety is generated by the classes of the  (finitely many) torus invariant divisors \cite[Lemma 15.1.8.]{CLS}, and on a smooth toric surface the intersection numbers of two torus invariant divisors is either zero or one \cite[p.98]{Fulton93}. So the conditions of 
Proposition~\ref{prop:rationalsurface} are fulfilled. 
  \end{proof}
  
  \begin{remark}
    The conditions on $X$ in Proposition~\ref{prop:rationalsurface} are also fulfilled for various non-toric surfaces e.g. for del Pezzo surfaces of degree at least $3$ or the blowup of $\bP^2$ in an arbitrary number of points on a line. Hence, for those surfaces the statement considered in \cite[Problem 2.5]{EinLazarsfeldMustopa} is also violated for every ample line bundle $\mathcal{L}$.
  \end{remark}

\bibliography{bib}
\bibliographystyle{alpha}
\end{document}